\documentclass[reqno,12pt]{amsart}

\usepackage{amsmath}
\usepackage{amsfonts}
\usepackage{amssymb}
\usepackage{eufrak}
\usepackage{amscd}
\usepackage{amsthm}
\usepackage{amstext}
\usepackage[all]{xy}

 \newcommand{\Q}{{\mathcal Q}}

   \theoremstyle{plain}
   \newtheorem{thm}{Theorem}
   \newtheorem{prop}[thm]{Proposition}
   \newtheorem{lem}[thm]{Lemma}
   \newtheorem{cor}[thm]{Corollary}
   \theoremstyle{definition}
   
   \newtheorem{defn}[thm]{Definition}
   
   \theoremstyle{remark}
   
   \newtheorem{remark}[thm]{Remark}

\author{V. Manuilov}

\date{}

\address{Moscow State University,
Leninskie Gory 1, Moscow, 
119991, Russia}

\email{manuilov@mech.math.msu.su}

\title{Approximately uniformly locally finite graphs}

\begin{document}

\maketitle

\begin{abstract}
Let $\Gamma$ be a locally finite graph, $L$ the normalized Laplacian of $\Gamma$. If $\Gamma$ is uniformy locally finite, i.e. if each vertex has no more than $d$ adjacent vertices, then the matrix of $L$ (with respect to the standard basis) has no more than $d+1$ non-zero entries in each row and in each column. We consider the class of locally finite graphs, for which the Laplacian can be approximated by matrices of this type with arbitrary $d$. We provide examples of locally finite graphs which are or are not in this class, and show that the graphs from this class share certain regularity property: vertices of high degree cannot have too many adjacent vertices of low degree. 

\end{abstract}

\section*{Introduction}

Let $\Gamma=(V,E)$ be a simple (i.e. without loops and multiple edges), locally finite (i.e. each vertex has a finite number of adjacent vertices), not necessarily connected graph. Here $V$ and $E$ are the sets of vertices and edges of $\Gamma$, respectively. For $v\in V$, let $d(v)$ denote the degree of $v$, i.e. the number of edges, which $v$ belongs to. For $v,w\in V$ we write $v\sim w$ if there is an edge connecting $v$ and $w$, i.e. if $v$ and $w$ are adjacent. Recall that $\Gamma$ is {\it uniformly} locally finite if there is a constant $d\in\mathbb N$ such that $d(v)\leq d$ for any $v\in V$. Our aim is to find a class of locally finite graphs, which are not uniformly loclally finite, but share, at least appoximatively, some of their features. 

Let $\mathcal D$ be the diagonal operator, $\mathcal Df(v)=d(v)f(v)$, where $f$ is a function on $V$, and let $\mathcal A$ be the adjacency operator, $\mathcal Af(v)=\sum_{w\sim v}f(w)$. Then  $\mathcal L=\mathcal D-\mathcal A$ is the Laplacian on $\Gamma$. It is a (not neccessarily bounded) operator on the Hilbert space $l^2(\Gamma)$ of square-summable functions on $V$ (for $f,g:V\to \mathbb C$, their inner product is $\langle f,g\rangle=\sum_{v\in V}\overline{f(v)}g(v)$).  If one prefers {\it bounded} operators, one has to pass to the normalized Laplacian $L=\mathcal D^{-1/2}\mathcal L\mathcal D^{-1/2}$. Then $L$ is a selfadjoint positive operator on $l^2(\Gamma)$ with $\|L\|\leq 2$. The formula for $L$ is 
$$
Lf(v)=f(v)-\sum_{w\sim v}\frac{1}{\sqrt{d(v)d(w)}}f(w). 
$$
It is convenient to use also the normalized adjacency operator $A=\mathcal D^{-1/2}\mathcal A\mathcal D^{-1/2}$, $L=I-A$. Our references on graph theory are \cite{Chung} and \cite{Bollobas}.

Denote by $\delta_v$ the characteristic function of the vertex $v\in V$. These functions form the standard orthonormal basis for $l^2(\Gamma)$. For an operator $B$ on $l^2(\Gamma)$ and for $v,w\in V$ we write $B_{vw}$ for the matrix entries of the matrix of $B$ with respect to the standard basis. As we will not use other bases for $l^2(\Gamma)$, we shall make no difference between operators and their matrices. 

Let $\mathbb B^{(d)}=\mathbb B^{(d)}(l^2(\Gamma))$ denote the set of bounded operators on $l^2(\Gamma)$, whose matrices with respect to the standard orthonormal basis of $l^2(\Gamma)$ (which consists of characteristic functions of vertices) satisfy the following property: each line and each column of the matrix of $L$ has no more than $d$ non-zero entries. Note that $\mathbb B^{(d)}$ is not a linear space, but if $A\in\mathbb B^{(d)}$ and $B\in\mathbb B^{(d')}$ then $A+B\in\mathbb B^{(d+d')}$. Set $\mathbb B^{(\infty)}=\cup_{d\in\mathbb N}\mathbb B^{(d)}$ and $\mathbb B=\overline{\mathbb B^{(\infty)}}$, where the closure is taken with respect to the operator norm on $l^2(\Gamma)$. It was shown in \cite{Manuilov} that $\mathbb B$ is a $C^*$-algebra.

If $\Gamma$ is uniformly locally finite, i.e. if $d(v)\leq d$ for any $v\in V$, then obviously $L\in\mathbb B{(d+1)}$. The opposite is also true: if $L\in\mathbb B^{(d)}$ then $\Gamma$ is uniformly locally finite. As $A=I+L$, the same holds true for the normalized adjacency operator as well.

Let us introduce a class of locally finite graphs, for which the normalized Laplacian can be approximated by operators from $\mathbb B^{(d)}$, $d\in\mathbb N$.

\begin{defn}
A graph $\Gamma$ is {\it approximately uniformly locally finite} (AULF) if $L\in\mathbb B$.

\end{defn}

We shall show that the class of AULF graphs is much wider than that of uniformly locally finite graphs, but AULF graphs still have some weakly regular properties, in particular, vertices of high degree cannot have too many adjacent vertices of low degree.
 
We write $\Gamma\in AULF$ if $\Gamma$ is an AULF graph.

\section{Conditions for being AULF}

Let $\Gamma=(V,E)$ be a graph, and let $\Gamma'=(V,E')$ be another graph with the same set of vertices and with additional edges, i.e. $E\subset E'$. In what follows, we write $d(v)$ for the degree of the vertex $v$ in $\Gamma$, and $d'(v)$ for its degree in $\Gamma'$. The condition $E\subset E'$ implies that $d(v)\leq d'(v)$ for any $v\in V$. We use $v\sim w$ if $v$ and $w$ are adjacent in $\Gamma$, and $v\sim' w$ if they are adjacent in $\Gamma'$.

\begin{thm}\label{add}
Suppose that
\begin{enumerate}
\item[{\rm (i)}]
$d'(v)\leq d(v)+1$ for any $v\in V$;
\item[{\rm (ii)}]
if $v\sim w$ then $d'(v)+d'(w)\leq d(v)+d(w)+1$ for any $v,w\in V$;
\item[{\rm (iii)}]
if $u\sim v$ and $v\sim w$ and $u\neq w$ then $d'(u)+d'(v)+d'(w)\leq d(u)+d(v)+d(w)+1$ for any $u,v,w\in V$.
\end{enumerate}
Then $\Gamma\in AULF$ iff $\Gamma'\in AULF$.

\end{thm}
\begin{proof}
Let us write $L$ and $L'$ for the normalized Laplacians on $\Gamma$ and on $\Gamma'$ respectively. The statement would follow if we show that $L'-L$ can be approximated by operators from $\mathbb B^{(\infty)}$. Write $L'-L$ as a sum of three operators, $L'-L=B+C+D$, where $D$ is diagonal, 
$$
D_{vw}=\left\lbrace\begin{array}{cl}\frac{1}{d(v)}-\frac{1}{d'(v)},&\mbox{if\ }w=v;\\0,&\mbox{otherwise,}\end{array}\right.
$$
$B$ and $C$ are off-diagonal,
$$
B_{vw}=\left\lbrace\begin{array}{cl}\frac{1}{\sqrt{d(v)d(w)}}-\frac{1}{\sqrt{d'(v)d'(w)}},&\mbox{if\ }w\sim v;\\0,&\mbox{otherwise,}\end{array}\right.
$$
$$
C_{vw}=\left\lbrace\begin{array}{cl}-\frac{1}{\sqrt{d'(v)d'(w)}},&\mbox{if\ }w\sim' v \mbox{\ but\ }w\nsim v;\\0,&\mbox{otherwise,}\end{array}\right.
$$
i.e. $B$ corresponds to the edges from $E$, and $C$ corresponds to the added edges.

It follows from (ii) that for any $v\in V$ there is no more than one vertex $w$ such that $w\sim' v$, but $w\nsim v$, hence the matrix $C$ has no more than one non-zero element in each line and in each column, so $C\in\mathbb B^{(1)}$. Also $D$, being diagonal, is patently in $\mathbb B^{(1)}$. Thus, it remains to approximate $B$. 

To this end, consider the set $V_1\subset V$, $V_1=\{v\in V:d'(v)=d(v)+1\}$. Let $v\in V_1$, $d(v)=d$, and let $P_v$ be the operator given by the matrix
$$
(P_v)_{uw}=\left\lbrace\begin{array}{rl}\frac{1}{\sqrt{d\cdot d(w)}}-\frac{1}{\sqrt{(d+1)d'(w)}},&\mbox{if\ }u=v \mbox{\ and\ }w\sim v;\\0,&\mbox{otherwise.}\end{array}\right.
$$
Note that by (ii), if $v\in V_1$ and $w\sim v$ then $d'(w)=d(w)$.

Set $B_v=P_v+P_v^*$. As $L'_{vw}-L_{vw}\neq 0$ only if either $v$ or $w$ lies in $V_1$, $B=\sum_{v\in V_1}B_v$ (we understand the infinite sum here as convergence with respect to the strong operator topology). 

Let $H_v\subset l^2(\Gamma)$ be the finitedimensional Hilbert space generated by the characteristic functions of $v$ and of all $w$ which are adjacent to $v$, and let $Q_v$ denote the projection onto $H_v$. Set $\overline{B}_v=Q_vB_v|_{H_v}$. If $v\notin V_1$ we assume that $\overline{B}_v=0$. Note that $H_v$ is invariant for $B_v$, and that $B_v$ is zero on the orthogonal complement of $H_v$, Note also that if $v,w\in V_1$ then $H_v$ and $H_w$ are orthogonal to each other. Indeed, if there is some $u\in V$ such that $u\sim v$ and $u\sim w$ then (iii) contradicts $v,w\in V_1$. Thus, the sum above is a direct sum: $B=\oplus_{v\in V_1}\overline{B}_v$. 

Let us estimate the norm of $\overline{B}_v$, where $d(v)=d$. If $w\sim v$ then $d(w)\geq 1$. As the matrix of $\overline{B}_v$ has only one non-zero row and only one non-zero column,  
\begin{eqnarray*}
\|\overline{B}_v\|&=&2\|P_v\|=2\sqrt{\sum_{w\sim v}\left(\frac{1}{\sqrt{d\cdot d(w)}}-\frac{1}{\sqrt{(d+1)d(w)}}\right)^2}\\
&\leq&2\sqrt{\sum_{w\sim v}\left(\frac{1}{\sqrt{d}}-\frac{1}{\sqrt{d+1}}\right)^2}=2\frac{\sqrt{d+1}-\sqrt{d}}{\sqrt{d(d+1)}}\sqrt{d}\leq \frac{1}{d}. 
\end{eqnarray*}

Take $\varepsilon>0$, and set 
$$
\widetilde B_v=\left\lbrace\begin{array}{cl}\overline{B}_v,&\mbox{if\ }d(v)<\varepsilon;\\0,&\mbox{otherwise,}\end{array}\right.
$$ 
$\widetilde B=\oplus_{v\in V_1}\widetilde B_v$. Then 
$$
\|B-\widetilde B\|=\sup_{v\in V_1}\|\overline{B}_v-\widetilde B_v\|<\varepsilon.
$$ 
Obviously, each $\widetilde B_v$ has no more that $d$ non-zero entries in each row and in each column, hence $\widetilde B\in \mathbb B^{([\frac{1}{\varepsilon}])}$.

\end{proof}

\begin{cor}\label{connect}
Let $\Gamma$ be not connected, with connected components $\Gamma_n$, and let $\Gamma'$ be the connected graph obtained from $\Gamma$ by connecting $\Gamma_n$ with $\Gamma_{n+1}$ by an edge. Then $\Gamma\in AULF$ iff $\Gamma'\in AULF$. 

\end{cor}
\begin{proof}
Apply Theorem \ref{add} not more than three times. Let $v_n$, $v'_n$ be the vertices in $\Gamma_n$ such that the edge of $\Gamma'$ connecting $\Gamma_n$ with $\Gamma_{n+1}$ connects $v'_n$ with $v_{n+1}$. If $v_n$ and $v'_n$ are not adjacent then it is enough to apply Theorem \ref{add} once. If some (or all) $v_n$ and $v'_n$ are adjacent then we should pass from $\Gamma$ to $\Gamma'$ in two steps: first, add the edges connecting $\Gamma_{2n-1}$ with $\Gamma_{2n}$, $n\in\mathbb N$, and, second, add the remaining edges. We cannot do both steps in one time, as this contradicts the condition (iii) of Theorem \ref{add}. The worst case is when $v_n=v'_n$ for all $n\in\mathbb N$. Then the first step is the same as above, but the condition (iii) prevents to make the second step, and we have to divide it again into two steps: first, add the edges connecting $v_{4n-2}$ with $v_{4n-1}$ for all $n\in\mathbb N$, and, second, add the remaining edges.

\end{proof}
 
The next example shows that if vertices with high degree have enough adjacent vertices with low degree then the graph is not AULF.

\begin{thm}\label{T1}
For a sequence $\{v_n\}_{n\in\mathbb N}$ of vertices in $\Gamma$, let $s_n^k$ denote the number of vertices $w$ adjacent to $v_n$ with $d(w)\leq k$. Suppose that
\begin{enumerate}
\item
$\lim_{n\to\infty}d_n=\infty$, where $d_n=d(v_n)$;
\item
$\lim\sup_{n\to\infty}\frac{s_n^k}{d_n}>0$ for some $k\in\mathbb N$.
\end{enumerate}
Then $\Gamma\notin AULF$.

\end{thm}
\begin{proof}
Assume that for any $\varepsilon>0$ there is some $r\in\mathbb N$ and some $B\in\mathbb B^{(r)}$ such that $\|L-B\|<\varepsilon$. Then $\|(L-B)\delta_{v_n}\|<\varepsilon$ for any $n\in\mathbb N$, where $\delta_{v}$ denotes the characteristic function of the vertex $v$. Take $\varepsilon<\lim\sup_{n\to\infty}\frac{s_n^k}{d_nk}$.

Since $B\in\mathbb B^{(r)}$, the vector $B\delta_{v_n}$ has not more than $r$ non-zero coordinates. On the other hand, the vector $L\delta_{v_n}$ has $s_n^k$ coordinates with the absolute value greater or equal to $\frac{1}{\sqrt{d_nk}}$. Therefore, $\|(L-B)\delta_n\|^2\geq (s_n^k-r)\frac{1}{d_nk}$ for any $n\in\mathbb N$, hence 
$$
\lim\sup_{n\to\infty}\|(L-B)\delta_{v_n}\|^2\geq \lim\sup_{n\to\infty}\frac{s_n^k-r}{d_nk}=\lim\sup_{n\to\infty}\frac{s_n^k}{d_nk}>\varepsilon. 
$$
This contradiction finishes the proof. 

\end{proof}

In particular, if $\Gamma\in AULF$, but not uniformly locally finite then, by adding enough degree one vertices and corresponing edges (leaves), we get a graph not in AULF.

\section{Examples}

\subsection{Unions of complete graphs}

Let $K_n$ be the complete graph with the set $V_n=\{1,2,\ldots,n\}$ of vertices, $\Gamma_K=\sqcup_{n\in\mathbb N}K_n$. 

\begin{prop}\label{proj}
$\Gamma_K\in AULF$.

\end{prop}
\begin{proof}
As $\Gamma_K$ is the disjoint union of the graphs $K_n$, the normalized Laplacian is the direct sum of normalized Laplacians $L_n$ of each $K_n$. Let $i,j\in V_n$. For each $n$, we have 
$$
(L_n)_{ij}=\left\lbrace\begin{array}{rl}\frac{1}{n},&\mbox{if\ }i=j;\\-\frac{1}{n},&\mbox{otherwise.}\end{array}\right.
$$ 
We shall approximate $\oplus_{n\in\mathbb N}\frac{2}{n}I-L_n$ instead of $\oplus_{n\in\mathbb N}L_n$ ($I$ is patently in $\mathbb B^{(1)}$). Note that $P_n=\frac{2}{n}I-L_n$ is the projection onto the vector $(1,\ldots,1)$. 

The further proof uses expander graphs and essentially is contained in \cite{Higson-Lafforgue-Skandalis}, where it is shown that the direct sum of one-dimensional projections $\oplus_{n\in\mathbb N}P_{m_n}$ lies in the uniform Roe algebra (hence can be approximated by operators in $\mathbb B^{(\infty)}$) for certain increasing sequences $\{m_n\}$ of sizes. It remains to note that one may take $m_n=n$ (cf. \cite{Manuilov}). Equivalently, the sequence $(P_n)_{n\in\mathbb N}$ of projections can be approximated by matrices in $\mathbb B^{(\infty)}$ uniformly in $n$, i.e. for any $\varepsilon>0$ there exists $d\in\mathbb N$ and matrices $C_n\in\mathbb B^{(d)}$ such that $\|P_n-C_n\|<\varepsilon$ for any $n\in\mathbb N$.

\end{proof}

Let $K_{k,l}$, $k\leq l$, denote the complete bipartite graph with the set of vertices $V_{k,l}=\{1,2,\ldots,k\}\cup\{1,2,\ldots,l\}$. For sequences $\kappa=(k_n)_{n\in\mathbb N}$, $\lambda=(l_n)_{n\in\mathbb N}$ of positive integers let $\Gamma_{\kappa,\lambda}=\sqcup_{n\in\mathbb N}K_{k_n,l_n}$ be the disjoint union of all $K_{k_n,l_n}$. 

\begin{prop}
If the sequence $\frac{l_n}{k_n}$, $n\in\mathbb N$, is uniformly bounded then $\Gamma_{\kappa,\lambda}\in AULF$, otherwise $\Gamma_{\kappa,\lambda}\notin AULF$.  

\end{prop}
\begin{proof}
Note that since $\Gamma_{\kappa,\lambda}=\sqcup_{n\in\mathbb N}K_{k_n,l_n}$, the normalized adjacency operator is the direct sum of those of each $K_{k_n,l_n}$, $A=\oplus_n A_n$, and each $A_n$ has the form $A_n=\left(\begin{matrix}0&B_n^*\\B_n&0\end{matrix}\right)$ with respect to the standard basis of $l^2(V_{k_n,l_n})=l^2(\{1,2,\ldots,k_n\})\oplus l^2(\{1,2,\ldots,l_n\})$. The matrix $B_n$ has $k_n$ columns and $l_n$ rows, and each entry is equal to $\frac{1}{\sqrt{k_nl_n}}$. $\Gamma_{\kappa,\lambda}\in AULF$ iff all $B_n$ can be approximated by matrices in $\mathbb B^{(\infty)}$ uniformly in $n$.

First, consider the case when there is some $c>0$ such that $\frac{l_n}{k_n}<c$ for any $n\in\mathbb N$. For each $n$, let $k_n(m_n-1)<l_n\leq k_nm_n$ for some $m_n\in\mathbb N$, $m_n<c+1$. Let $P_{k_n}$ be the square $k_n{\times}k_n$-matrix with all entries equal to $\frac{1}{k_n}$, let $Q_n$ be the matrix of the projection onto the first $l_n-k_n(m_n-1)$ coordinates in $\mathbb C^{k_n}$, and let $R_{k_n}$ denote the first $l_n-k_n(m_n-1)$ rows of $P_{k_n}$, or, equivalently, of $Q_nP_{k_n}$. Then the matrix $B_n$ can be written as $B_n=\sqrt{\frac{k_n}{l_n}}\left(\begin{smallmatrix}P_{k_n}\\ \vdots\\P_{k_n}\\R_{k_n}\end{smallmatrix}\right)$.  

We know from Proposition \ref{proj} that for any $\varepsilon>0$ there exists $d\in\mathbb N$ and matrices $C_n\in\mathbb B^{(d)}$ such that $\|P_{k_n}-C_{k_n}\|<\varepsilon$. Then $\|Q_nP_{n_k}-Q_nC_{k_n}\|<\varepsilon$. Then
$$
\left\|\left(\begin{smallmatrix}P_{k_n}\\ \vdots\\P_{k_n}\\Q_nP_{k_n}\end{smallmatrix}\right)-\left(\begin{smallmatrix}C_{k_n}\\ \vdots\\C_{k_n}\\Q_nC_{k_n}\end{smallmatrix}\right)\right\|\leq m_n\|P_{k_n}-C_{k_n}\|<(c+1)\varepsilon.
$$   
As $C_{k_n}\in \mathbb B^{(d)}$, the matrix $\left(\begin{matrix}C_{n_k}\\\vdots\\C_{n_k}\end{matrix}\right)$ lies in $\mathbb B^{((m_n+1)d)}$, thus the matrices $B_n$ can be uniformly approximated by matrices from $\mathbb B^{(\infty)}$.

Now turn to the second case: assume that there is no upper bound for $\frac{l_n}{k_n}$. Passing to a subsequence, if necessary, we may assume that $\lim_{n\to\infty}\frac{l_n}{k_n}=\infty$. Assume that $B_n$ can be approximated by matrices from $\mathbb B^{(\infty)}$: take $\varepsilon<1/2$, and find $d\in\mathbb N$ and matrices $C_n\in\mathbb B^{(d)}$ such that $\|B_n-C_n\|<\varepsilon$ for any $n\in\mathbb N$. The number of columns in $B_n$ and in $C_n$ equals $k_n$, and as $C_n$ has no more than $d$ non-zero entries in each column, the matrix $C_n$ has at least $l_n-k_nd$ zero rows. Let $S_n$ be the projection onto the coordinates that correspond to the zero rows. Then $\|B_n-C_n\|\geq \|S_n(B_n-C_n)\|=\|B'_n\|$, where $B'_n$ is the $k_n{\times}(l_n-k_nd)$-matrix with each entry equal to $\frac{1}{\sqrt{k_nl_n}}$. The latter norm is easily computable: $\|B'_n\|=\frac{l_n-k_nd}{l_n}$, and is greater than 1/2 for sufficiently great $n$ --- a contradiction.     

Note that the case when $(k_n)_{n\in\mathbb N}$ is bounded, and $\lim_{n\to\infty}l_n=\infty$ easily follows from Theorem \ref{T1}.

\end{proof} 

Remark that although the graphs in the two previous examples were not connected, one can use Corollary \ref{connect} to produce similar connected examples.

\subsection{Unions of regular graphs}

Let $\Gamma_n$ be an $n$-regular graph with $m_n$ vertices, and let $\Gamma=\sqcup_{n\in\mathbb N}\Gamma_n$. If $\Gamma_n$ are complete (i.e. when $m_n=n$) then, by Proposition \ref{proj}), $\Gamma\in AULF$. We do not know if $\Gamma$ is AULF in general. This seem to require fine norm estimates. But we can show that $\Gamma\in AULF$ if the sequence $(m_n)_{n\in\mathbb N}$ grows very slowly.  

\begin{prop}
Let $\Gamma_{n,m}$ denote an $n$-regular graph on $m$ vertices. Let $(m_n)_{n\in\mathbb N}$ be a sequence such that $m_n\geq n$ for any $n\in\mathbb N$, and $\lim_{n\to\infty}\frac{m_n}{n}=1$. Then $\Gamma=\sqcup_{n\in\mathbb N}\Gamma_{n,m_n}\in AULF$.

\end{prop}
\begin{proof}
For a graph $G$, let $\overline G$ denote the graph, which is the complement to $G$, i.e. it has the same set of vertices as $G$, and two vertices $v$ and $w$ are adjacent in $\overline{G}$ iff they are not adjacent in $G$. Note that if $G$ is regular then $\overline G$ is regular. For the normalized adjacency operators we have $nA_{\Gamma_{n,m}}=mA_{K_m}-(m-n)A_{\overline{\Gamma}_{n,m}}$, where $K_m$ is the complete graph with the same set of vertices. By Proposition \ref{proj}, $\sqcup_{m\in\mathbb N}K_m\in AULF$, so it remains to approximate $\oplus_{n\in\mathbb N}\frac{m_n-n}{n}A_{\overline{\Gamma}_{n,m_n}}$ by operators from $\mathbb B^{(\infty)}$. This would follow if we show that $\lim_{n\to\infty}\frac{m_n-n}{n}\|A_{\overline{\Gamma}_{n,m_n}}\|=0$. Passing to non-normalized adjacency matrices, this is equivalent to $\lim_{n\to\infty}\frac{1}{n}\|\mathcal A_{\overline{\Gamma}_{n,m_n}}\|=0$, where all non-zero entries in $\mathcal A$ equal 1 (and each row and column of $\mathcal A$ has $m_n-n$ non-zero entries). This would follow from the estimate $\|\mathcal A_{\overline{\Gamma}_{n,m_n}}\|\leq m_n-n$, which is proved in the following Lemma, which should be folklore, but we couldn't find a reference.

\end{proof}

\begin{lem}
Let $A$ be a matrix with all entries equal to 0 or 1, and let each row and each column contain exactly $d$ 1's. Then $\|A\|\leq d$.

\end{lem}

\begin{proof}
Let the dimension of $A$ be $m$. Then there exists a bipartite graph $G=(V,E)$ with two sets of vertices, $V_1=\{i_1,\ldots,i_m\}$ and $V_2=\{j_1,\ldots,j_m\}$, $V=V_1\sqcup V_2$, such that $A_{kl}=1$ iff $i_k\sim j_l$. By assumption, $G$ is $d$-regular. We shall argue by induction in $d$. The statement is obviously true for $d=1$. We have to pass from $d-1$ to $d$. 

Recall that a perfect matching in $G$ is a permutation $\sigma$ of $\{1,\ldots,m\}$ such that $j_{\sigma(r)}\sim i_r$ for any $r\in\{1,\ldots,m\}$. It is known that regular bipartite graphs have perfect matchings \cite{Bollobas}. Let $\sigma$ be one of them. Let $G'$ be the bipartite graph with the same set of vertices, but with the edges that connect $j_{\sigma(r)}$ with $i_r$, $r\in\{1,\ldots,m\}$, removed. If $G$ is $d$-regular, then $G'$ is $d-1$-regular. Write $A=A_\sigma+A_{d-1}$, where 
$$
(A_\sigma)_{kl}=\left\lbrace\begin{array}{cl}1,&\mbox{if\ }l=\sigma(k);\\0,&\mbox{otherwise.}\end{array}\right.
$$
Then $A_{d-1}$ has $d-1$ non-zero entries in each row and in each column, so, by induction assumption, $\|A_{d-1}\|\leq d-1$. As regards $A_\sigma$, note that it is the operator of permutation of the vectors of the basis, hence is unitary, therefore, $\|A_\sigma\|=1$.

\end{proof}

\subsection{Trees}

Now let us consider some trees.
Let $\kappa=(k_0,k_1,k_2,\ldots)$ be a sequence of positive integers. Define a rooted tree by assuming that the root $v_0$ has $d(v_0)=k_0$, each vertex at level 1 (i.e. adjacent to the root) has $d(v)=k_1+1$, each vertex at level 2 (i.e. having distance 2 from the root) has $d(v)=k_2+1$, etc., i.e. each vertex at level $n$ has $d(v)=k_n+1$. Denote this tree by $T(\kappa)$.

\begin{prop}
Suppose that there exists a constant $C>0$ such that $k_n\leq Ck_{n+1}$ for any $n\in\mathbb N$. Then $T(\kappa)\in AULF$.

\end{prop}
\begin{proof}
Let us define the set of all vertices of level $n$ by $V_n$. Thus, $V=\sqcup_{n=0}^\infty V_n$ is the set of vertices of $T(\kappa)$. Let us linearly order the set $V$ in such a way that $v\leq w$ if $v\in V_n$, $w\in V_m$, and $n<m$. Decompose the normalized adjacency operator $A$ as $A=B+B^*$, where the operator $B$ is lower-triangular:
$$
B_{vw}=\left\lbrace\begin{array}{cl}A_{vw},&\mbox{if\ }v\geq w;\\0,&\mbox{otherwise.}\end{array}\right.
$$

Let us show that $\langle B\delta_v,B\delta_w\rangle=0$ if $w\neq v$. Note that  
\begin{equation}\label{sum1}
\langle B\delta_v,B\delta_w\rangle=\sum_{u\in V}B_{uv}B_{uw}.
\end{equation}
Consider two cases: (a) $v\in V_n$, $w\in V_m$, $m\neq n$; and (b) $v,w\in V_n$.
In the first case $B_{uv}\neq 0$ only when $u\in V_{n+1}$, and $B_{uw}\neq 0$ only when $u\in V_{m+1}$, hence each summand in (\ref{sum1}) is zero. In the second case $B_{uv}B_{uw}\neq 0$ only if $u\in V_{n+1}$ and $u$ is adjacent to both $v$ and $w$, which cannot happen in a tree.

Let $\xi\in l^2(T)$, $\xi=\sum_{v\in V}\xi_v\delta_v$, $\sum_{v\in V}|\xi_v|^2=1$. Then 
$$
\|B\xi\|^2=\sum_{v,w\in V}\bar\xi_v\xi_w\langle B\delta_v,B\delta_w\rangle=\sum_{v\in V}|\xi_v|^2\|B\delta_v\|^2.
$$ 
It follows that $\|B\|=\sup_{v\in V}\|B\delta_v\|$.

For a subset $U\subset\mathbb N$, define $B^U$ by $B^U\delta_v=\left\lbrace\begin{array}{cl}B\delta_v,&\mbox{if\ }v\in V_n, \mbox{\ where\ }n\in U;
\\0,&\mbox{otherwise.}\end{array}\right.$

The same argument as above shows that 
$
\|B^U\|=\sup_{v\in \sqcup_{n\in U} V_n}\|B\delta_v\|.
$ 

It is easy to see that if $v\in V_n$, $n>0$, then $B\delta_v$ has $d(v)-1=k_n$ non-zero coordinates (corresponding to the descendents of $v$ in $V_{n+1}$), each of which equals $\frac{1}{\sqrt{(k_n+1)(k_{n+1}+1)}}$, hence $\|B\delta_v\|^2=\frac{k_n}{(k_n+1)(k_{n+1}+1)}\leq \frac{1}{k_{n+1}}$.

Take $\varepsilon>0$, and define the subset $U\subset\mathbb N$: $n\in U$ iff $k_{n+1}<\frac{1}{\varepsilon^2}$. Then $\|B-B^U\|=\sup_{v\in \sqcup_{n\notin U} V_n}\|B\delta_v\|<\varepsilon$. Each row in $B^U$ contains only one non-zero entry, so we have to evaluate the number of non-zero entries in each column of $B^U$. The vector $B\delta_v$ is the column corresponding to the vertex $v$, so it has $d(v)=k_n+1$ equal non-zero entries. By assumption, $k_n\leq Ck_{n+1}$, so $k_n+1\leq Ck_{n+1}+1\leq \frac{C}{\varepsilon^2}+1$, hence $B\in \mathbb B^{(s)}$, where $s=\frac{C}{\varepsilon^2}+1$. Since $A=B+B^*$, $A\in\mathbb B^{(2s)}$. 

\end{proof}

\begin{remark}
Note that the existence of $C$ such that $k_n\leq Ck_{n+1}$ is essential. Suppose that there exists some $k\in\mathbb N$ such that $k_{3n-1}\leq k$, $k_{3n-2}\leq k$ for any $n\in\mathbb N$, but $\lim_{n\to\infty}k_{3n}=\infty$. Then, by Theorem \ref{T1}, $T(\kappa)\notin AULF$. Indeed, passing to a subsequence $(n_i)_{i\in\mathbb N}$, we can find $1\leq l\leq k$ such that the descendants $w\in V_{3n_i+1}$ of $v\in V_{3n_i}$ have $d(w)=l+1$ for any $i\in\mathbb N$ (due to the finite number of possibilities for $d(w)$). The number of these descendants equals $s_{3n_i}^{l+1}=k_{3n_i}$, hence $\lim_{i\to\infty}\frac{s_{3n_i}^{l+1}}{k_{3n_i}+1}=1$.  

\end{remark}

\end{document}